\documentclass{amsart}%
\usepackage{amssymb}
\usepackage{amsfonts}
\usepackage{geometry}%
\usepackage{amsmath}%
\usepackage{graphicx}
\providecommand{\U}[1]{\protect\rule{.1in}{.1in}}
\newtheorem{theorem}{Theorem}
\theoremstyle{plain}

\newtheorem{definition}[theorem]{Definition}

\newtheorem{lemma}[theorem]{Lemma}

\numberwithin{equation}{section}
\begin{document}
\title[Schr\"{o}dinger equation]{Probabilistic Fourier extension and the Schr\"{o}dinger equation}
\author{Lyudmila Korobenko}
\address{Reed College\\
Portland, Oregon, USA}
\email{lkoroben@math.mcmaster.ca}
\author[E. T. Sawyer]{Eric T. Sawyer$^{\dagger}$}
\address{Eric T. Sawyer, Department of Mathematics and Statistics\\
McMaster University\\
1280 Main Street West\\
Hamilton, Ontario L8S 4K1 Canada}
\email{sawyer@mcmaster.ca}
\thanks{$\dagger$ Research supported in part by a grant from the National Science and
Engineering Research Council of Canada.}
\date{\today }

\begin{abstract}
The probabilistic Fourier extension theorem for the paraboloid in $d$
dimensions (see \cite{Saw7}\ for the sphere and the proof there holds for a
smooth compact piece of the paraboloid) shows that%
\[
\mathbb{E}_{\pm}\left\Vert \mathcal{E}_{S}f_{\pm}\right\Vert _{L^{p}\left(
\mathbb{R}^{d}\right)  }\leq C\left\Vert f\right\Vert _{L^{p}\left(  U\right)
}\ ,\ \ \ \ \ \left(  p,q\right)  \in\Gamma_{\operatorname*{Knapp}},
\]
where $\mathcal{E}_{S}f_{\alpha}\left(  \xi\right)  \equiv\int_{U}%
e^{-i\Phi\left(  x\right)  \cdot\xi}f_{\alpha}\left(  x\right)  dx$,
$\Phi\left(  x\right)  =\left(  \sqrt{1-\left\vert x\right\vert ^{2}%
},x\right)  $, $f_{\pm}$ is a `martingale transform' of $f$, and
\[
\Gamma_{\operatorname*{Knapp}}\equiv\left\{  \left(  p,q\right)  :q=\frac
{n+1}{n-1}p^{\prime},2\leq p<\frac{2n}{n-1}\right\}  .
\]

We give applications of this theorem to the global initial value problem for
the Schr\"{o}dinger equation,%
\[
\left\{
\begin{array}
[c]{cccc}%
i\frac{\partial u}{\partial t}\left(  t,x\right)  -\bigtriangleup_{x}u\left(
t,x\right)  & = & 0 & \left(  t,x\right)  \in\mathbb{R}\times\mathbb{R}%
^{d-1}\\
u\left(  0,x\right)  & = & f\left(  x\right)  & x\in\mathbb{R}^{d-1}%
\end{array}
\right.  ,
\]
in which we obtain sharp averaged Sobolev regularity results over certain
`martingale transforms' of the intial data $f$. The deterministic analogue in
three dimensions has been obtained recently by OpenAI \cite{Ope}.

\end{abstract}
\maketitle
\tableofcontents

\section{Introduction}

The Fourier extension conjecture in $d\geq2$ dimensions arose from unpublished
work of E. Stein in 1967, see e.g. \cite[see the Notes at the end of Chapter
IX, p. 432, where Stein proved the restriction conjecture for $1\leq
p<\frac{4d}{3d+1}$]{Ste2} and \cite{Ste},%
\begin{equation}
\left\Vert \widehat{f\sigma_{d-1}}\right\Vert _{L^{q}\left(  \mathbb{R}%
^{d}\right)  }\leq C_{p}\left\Vert f\right\Vert _{L^{p}\left(  \sigma
_{d-1}\right)  },\ \text{for }f\in L^{p}\left(  \sigma_{d-1}\right)
,\label{extension}%
\end{equation}
for%
\begin{align}
& \left(  \frac{1}{p},\frac{1}{q}\right)  \in\left[  0,1\right]  ^{2}\text{
lying strictly below the line }\frac{1}{q}=\frac{d-1}{2d}%
,\label{Knapp conditions}\\
& \text{and on or below the Knapp line through }\left(  \frac{d-1}{2d}%
,\frac{d-1}{2d}\right)  \text{ and }\left(  1,0\right)  ,\nonumber
\end{align}
and where $\sigma_{d-1}$ is surface measure on $S$ and where $S$ is either the
sphere $\mathbb{S}^{d-1}$ or a smooth compactly supported measure on the
paraboloid $\mathbb{P}^{d-1}$. The case $d=2$ of the Fourier extension
conjecture for the circle was proved over half a century ago by C. Fefferman
\cite{Fef}, see also A. Zygmund \cite{Zyg} and L. Carleson and P. Sj\"{o}lin
\cite{CaSj}. Just recently, OpenAI has claimed to prove the diagonal
inequality $p=q$ in the case of $d=3$ dimensions with $\mathsf{L}%
\exists\forall\mathsf{N}$ certification \cite{Ope}.

While the conjecture remains open in all higher dimensions $d>3$, with
diagonal bounds inching ever closer to the critical index $\frac{2d}{d-1}$,
the \emph{probabilistic} weakening of the conjecture with smooth Alpert
wavelets has been proved in all dimensions for the diagonal case
$p=q>\frac{2d}{d-1}$ in \cite{Saw7}. Here the frame of smooth Alpert
projections $\left\{  \bigtriangleup_{I;\kappa}^{\eta}\right\}  _{I\in
\mathcal{G}}$ associated with a dyadic grid $\mathcal{G}$ has been constructed
in \cite{Saw7} with $\bigtriangleup_{I;\kappa}^{\eta}f=\left\langle
S^{-1}f,h_{I;\kappa}\right\rangle h_{I;\kappa}^{\eta}$, where $S$ is a bounded
invertible map on $L^{p}$,\ and the full probabilistic Fourier extension
conjecture is the inequality%
\begin{equation}
\mathbb{E}_{2^{\mathcal{G}}}^{\mu}\left\Vert \mathcal{E}\mathbf{1}_{U}%
\sum_{I\in\mathcal{G}}\pm\bigtriangleup_{I;\kappa}^{\eta}f\right\Vert
_{L^{q}\left(  \mathbb{R}^{d}\right)  }\leq C\left\Vert f\right\Vert
_{L^{p}\left(  B\left(  0,\frac{1}{2}\right)  \right)  },\label{prob ext}%
\end{equation}
for $p$ and $q$ satisfying (\ref{Knapp conditions}), and where $\sigma_{d-1} $
is a smooth compactly supported measure on the sphere $\mathbb{S}^{d-1}$ or
the paraboloid $\mathbb{P}^{d-1}$, $\widehat{\mu}\equiv\int_{\mathbb{R}^{d}%
}e^{-ix\cdot\xi}d\mu\left(  x\right)  $ denotes the Fourier transform of a
measure $\mu$, and where the expectation $\mathbb{E}_{2^{\mathcal{G}}}^{\mu}$
is taken over all choices of $\pm$ for each $I\in\mathcal{G}$. Finally,
$\mathcal{E}_{S}$ is the Fourier extension operator
\[
\mathcal{E}_{S}f\left(  \xi\right)  =\int_{\mathbb{R}^{d-1}}e^{-2\pi i\xi
\cdot\Phi_{S}\left(  x\right)  }f\left(  x\right)  dx,
\]
where $\Phi_{S}\left(  x\right)  =\left(  x,\phi_{S}\left(  x\right)  \right)
$ and $\operatorname*{Graph}\phi_{S}\subset S$.

The theorem in \cite{Saw7} proves (\ref{prob ext}) in the case of the sphere,
but the proof there carries over easily to a smooth compact piece of the
paraboloid. We will assume $S$ is such a piece of the paraboloid in this
paper. Using Khintchine's inequality, the probabilistic inequality
(\ref{prob ext}) can be equivalently reformulated as a Fourier square function
inequality,%
\begin{equation}
\left\Vert \left(  \sum_{I\in\mathcal{G}}\left\vert \mathcal{E}\mathbf{1}%
_{U}\bigtriangleup_{I;\kappa}^{\eta}f\right\vert ^{2}\right)  ^{\frac{1}{2}%
}\right\Vert _{L^{q}\left(  \mathbb{R}^{d}\right)  }\leq C\left\Vert
f\right\Vert _{L^{p}\left(  B\left(  0,\frac{1}{2}\right)  \right)
}.\label{square function}%
\end{equation}
It was pointed out in \cite{Saw7} that using the square function inequality
(\ref{square function}), the probabilistic analogue (\ref{prob ext}) fails for
the same pairs $\left(  p,q\right)  $ that (\ref{extension}) is currently
known to fail for - namely those not satisfying (\ref{Knapp conditions}). The
Knapp arc
\[
A_{\operatorname*{Knapp}}\equiv\left\{  \left(  p,q\right)  :\frac{1}{p}%
+\frac{d+1}{d-1}\frac{1}{q}=1\text{ and }2\leq p<\frac{2d}{d-1}\right\}
\]
represents the only set of pairs $\left(  p,q\right)  $ for which it is
\emph{not} known if (\ref{prob ext}) holds or not. Note also that the Knapp
arc can be described as the set of points $\left(  p,q\right)  $ such that
$\left(  \frac{1}{p},\frac{1}{q}\right)  $ lies on the line segment
$S_{\operatorname*{Knapp}}$ that is the subset of the line $x+\frac{d+1}%
{d-1}y=1$ lying strictly between the points corresponding to $p=2$ and
$p=\frac{2d}{d-1}$. The probabilistic Fourier extension theorem will be used
to obtain probabilistic regularity results for the global Schr\"{o}dinger equation.

\section{The global initial value problem for the Schr\"{o}dinger equation}

Denote by $\mathcal{S}\left(  \mathbb{R}^{d}\right)  $ the space of rapidly
decreasing functions on $\mathbb{R}^{d}$, and let $\mathcal{S}\left(
\mathbb{R}^{d}\right)  ^{\ast}$ be the space of tempered distributions.
Consider the global initial value problem for the Schr\"{o}dinger equation,%
\begin{equation}
\left\{
\begin{array}
[c]{cccc}%
i\frac{\partial u}{\partial t}\left(  t,x\right)  -\bigtriangleup_{x}u\left(
t,x\right)  & = & 0 & \left(  t,x\right)  \in\mathbb{R}\times\mathbb{R}%
^{d-1}\\
u\left(  0,x\right)  & = & f\left(  x\right)  & x\in\mathbb{R}^{d-1}%
\end{array}
\right.  ,\label{IVP Schrodinger}%
\end{equation}
where $f\in\mathcal{S}\left(  \mathbb{R}^{d-1}\right)  ^{\ast}$ is a tempered
distribution on $\mathbb{R}^{d-1}$. We will now show the probabilistic Fourier
extension theorem can be used to prove a probabilistic regularity theorem for
(\ref{IVP Schrodinger}).

But first, for the sake of completeness, we include the proof of the following
well known connection between the Schrodinger equation and Fourier extensions.
Let $\Phi\left(  \tau,\xi\right)  \equiv\left(  2\pi\left\vert \xi\right\vert
^{2},\xi\right)  $ be the standard parameterization of $\mathbb{P}^{d-1}$, and
let $\mathcal{F}_{d}$ denote the Fourier transform in Euclidean space
$\mathbb{R}^{d}$.

\begin{theorem}
Let $f\in\mathcal{S}\left(  \mathbb{R}^{d-1}\right)  ^{\ast}$ be a tempered
distribution in $\mathbb{R}^{d-1}$. Then $u_{f}\equiv\mathcal{F}_{d}^{-1}%
\Phi_{\ast}\mathcal{F}_{d-1}f$ is a tempered distribution and is a solution to
the Schr\"{o}dinger initial value problem (\ref{IVP Schrodinger}).
\end{theorem}

\begin{proof}
Denote by $\sigma\left(  \tau,\xi\right)  \equiv\left(  -2\pi\tau+4\pi
^{2}\left\vert \xi\right\vert ^{2}\right)  $ the symbol of the differential
operator $\mathcal{L}\equiv i\frac{\partial u}{\partial t}-\bigtriangleup_{x}$
on $\mathbb{R\times R}^{d-1}$. Then the zero set $\mathcal{Z}\left(
\sigma\right)  $ of $\sigma$ is the paraboloid%
\[
\mathbb{P}^{d-1}\equiv\left\{  \left(  \tau,\xi\right)  \in\mathbb{R}%
\times\mathbb{R}^{d-1}:\tau=2\pi\left\vert \xi\right\vert ^{2}\right\}  .
\]

In particular we see that%
\[
\mathcal{F}_{d}u_{f}=\Phi_{\ast}\mathcal{F}_{d-1}f
\]
is supported in the paraboloid $\mathbb{P}^{d-1}$, and it follows that
\[
\sigma\left(  \tau,\xi\right)  \mathcal{F}_{d}u_{f}\left(  \tau,\xi\right)
=\sigma\left(  \tau,\xi\right)  \Phi_{\ast}\mathcal{F}_{d-1}f\left(
\xi\right)  =0,
\]
since $\operatorname*{Supp}\Phi_{\ast}\mathcal{F}_{d-1}f\subset\mathbb{P}%
^{d-1}=\mathcal{Z}\left(  \sigma\right)  $. Thus
\[
0=\mathcal{F}_{d}^{-1}\left[  \sigma\left(  \tau,\xi\right)  \mathcal{F}%
_{d}u_{f}\right]  =\mathcal{LF}_{d}^{-1}\mathcal{F}_{d}u_{f}=\mathcal{L}u_{f},
\]
shows that the tempered distribution $u_{f}$ satisfies the Schr\"{o}dinger
equation in the first line of (\ref{IVP Schrodinger}).

Moreover,
\begin{align*}
u\left(  t,x\right)   & =\mathcal{F}_{d}^{-1}\Phi_{\ast}\mathcal{F}%
_{d-1}f\left(  t,x\right)  =\int_{\mathbb{R}^{d-1}}e^{2\pi i\left(
2\pi\left\vert \xi\right\vert ^{2}t+\xi\cdot x\right)  }\mathcal{F}%
_{d-1}f\left(  \xi\right)  d\xi\\
& \Longrightarrow u\left(  0,x\right)  =\int_{\mathbb{R}^{d-1}}e^{2\pi
i\xi\cdot x}\mathcal{F}_{d-1}f\left(  \xi\right)  d\xi=f\left(  x\right)  ,
\end{align*}
by the Fourier inversion formula, and so we see that $u_{f}$ satisfies the
initial condition in the second line of (\ref{IVP Schrodinger}), and this
completes the proof in part (1) that $u_{f}$ solves (\ref{IVP Schrodinger}).
\end{proof}

It is well known, see e.g. \cite{Tao}, that since the Fourier extension
conjecture (\ref{extension}) is known to be true when $p=2$ in all dimensions
by the Stein-Tomas theorem \cite{Tom}, we have the following $L^{2}$
regularity result for the initial value problem (\ref{IVP Schrodinger}),%
\begin{align*}
& \left\Vert u_{f}\right\Vert _{L^{q}\left(  \mathbb{R}^{d}\right)
}=\left\Vert \mathcal{F}_{d}^{-1}\Phi_{\ast}\mathcal{F}_{d-1}f\right\Vert
_{L^{q}\left(  \mathbb{R}^{d}\right)  }=\left\Vert \mathcal{EF}_{d-1}%
f\right\Vert _{L^{q}\left(  \mathbb{R}^{d}\right)  }\leq C_{d,\sigma
,q}\left\Vert \mathcal{F}_{d-1}f\right\Vert _{L^{2}\left(  \mathbb{R}%
^{d}\right)  }=C_{d,\sigma,q}\left\Vert f\right\Vert _{L^{2}\left(
\mathbb{R}^{d}\right)  },\\
& \ \ \ \ \ \ \ \ \ \ \ \ \ \ \ \ \ \ \ \ \text{for all }f\in L^{2}\left(
\mathbb{R}^{d}\right)  \text{ provided (\ref{Knapp conditions}) holds},
\end{align*}
upon appealing to Plancherel's theorem in the final equality. To obtain an
analogous $L^{p}$ bound for $1<p\leq2$, we can of course use Hausdorff-Young
in place of Plancherel, but if the extension inequality is only known on the
diagonal $p=q$ for compactly supported $f$, then a stronger Sobolev norm will
be needed on $f$ in order to sum up over positive scales. This is in fact the
case for the probabilistic inequality (\ref{prob ext}), and dictates the form
of the regularity inequality we obtain in Theorem \ref{Thm reg} below. This
strategy was also used to obtain the deterministic analogue in $d=3$
dimensions in \cite[see the proof of Corollary 1.2 in Section 7]{Ope}.

We momentarily consider the question of finding Banach spaces $B$, that are
large enough to contain the Schwarz space $\mathcal{S}\left(  \mathbb{R}%
^{d-1}\right)  $ of rapidly decreasing functions, such that the solution
$u_{f}$ has near maximal decay at $\infty$ for $f\in B$, i.e. lies in
$L^{q}\left(  \mathbb{R}^{d}\right)  $ for as small an index $q$ as possible.
If $f$ is the Fourier transform of a smooth cutoff function, then $\left\Vert
u_{f}\right\Vert _{L^{\frac{2d}{d-1}}\left(  \mathbb{R}^{d}\right)  }=\infty$,
which shows that we can only obtain $u_{f}\in L^{q}\left(  \mathbb{R}%
^{d}\right)  $ for $q>\frac{2d}{d-1}$ when $B$ contains $\mathcal{S}\left(
\mathbb{R}^{d-1}\right)  $. The following \emph{probabilistic} regularity
theorem shows that one can obtain such a bound with $q$ arbitrarily close to
$\frac{2d}{d-1}$, by taking $B$ to be an $L^{p}$-Sobolev space with
arbitrarily little smoothness for appropriate $p$, and by taking averages over
all `martingale transforms' of the data $f$ in (\ref{IVP Schrodinger}).

\begin{definition}
Let $1<p\leq2$. For $f\in L^{p}\left(  \mathbb{R}^{d-1}\right)  $, the
Hausdorff-Young inequality implies $\widehat{f}\in L^{p^{\prime}}\left(
\mathbb{R}^{d-1}\right)  $, and so $\widehat{f}$\ has an Alpert decomposition
\[
\widehat{f}=\sum_{I\in\mathcal{G}}\left\langle S^{-1}\widehat{f},h_{I;\kappa
}\right\rangle h_{I;\kappa}^{\eta}%
\]
both pointwise a.e. and in $L^{p^{\prime}}$ norm. We define the \emph{Fourier}
`martingale transform' $f_{\alpha}$ of $f$ by
\[
f_{\alpha}\equiv\sum_{I\in\mathcal{G}}\alpha_{I}\left\langle S^{-1}\widehat
{f},h_{I;\kappa}\right\rangle h_{I;\kappa}^{\eta}\ ,\ \ \ \ \ \alpha=\left\{
\alpha_{I}\right\}  _{I\in\mathcal{G}}\in2^{\mathcal{G}}\equiv\left\{
-1,1\right\}  ^{\mathcal{G}},
\]
where $\alpha$ denotes a random choice of sign $\pm$.
\end{definition}

Note that if $\alpha_{0}$ is the sequence of all plus signs, then
$f_{\alpha_{0}}=\widehat{f}$, which explains the terminology Fourier
`martingale transform'. Note also that the solution $u_{f_{\alpha}}$ to
(\ref{IVP Schrodinger}) with initial data $f_{\alpha}$, satisfies
\begin{align*}
& u_{f_{\alpha}}=\mathcal{F}_{d}^{-1}\Phi_{\ast}\mathcal{F}_{d-1}f_{\alpha
}=\mathcal{F}_{d}^{-1}\Phi_{\ast}\mathcal{F}_{d-1}\sum_{I\in\mathcal{G}}%
\alpha_{I}\left\langle S^{-1}\widehat{f},h_{I;\kappa}\right\rangle
h_{I;\kappa}^{\eta}\\
& =\sum_{I\in\mathcal{G}}\alpha_{I}\left\langle S^{-1}\widehat{f},h_{I;\kappa
}\right\rangle \mathcal{F}_{d}^{-1}\Phi_{\ast}\mathcal{F}_{d-1}h_{I;\kappa
}^{\eta}=\sum_{I\in\mathcal{G}}\alpha_{I}\left\langle S^{-1}\widehat
{f},h_{I;\kappa}\right\rangle u_{h_{I;\kappa}^{\eta}},\\
& \ \ \ \ \ \ \ \ \ \ \ \ \ \ \ \text{for }\alpha\in2^{\mathcal{G}}%
\equiv\left\{  -1,1\right\}  ^{\mathcal{G}}.
\end{align*}
Finally, we denote by $W^{\sigma,p}\left(  \mathbb{R}^{d-1}\right)  $ the
(inhomogeneous) $L^{p}$ Sobolev space of order $\sigma>0$ on $\mathbb{R}%
^{d-1}$ normed by%
\[
\left\Vert f\right\Vert _{W^{\sigma,p}\left(  \mathbb{R}^{d-1}\right)
}=\left\Vert \left(  I-\bigtriangleup\right)  ^{\frac{\sigma}{2}}f\right\Vert
_{L^{p}\left(  \mathbb{R}^{d-1}\right)  }.
\]

\begin{theorem}
\label{Thm reg}Let $1<p\leq2$ and $\sigma>0$. Suppose $f\in W^{\sigma
,p}\left(  \mathbb{R}^{d}\right)  $ and $f_{\alpha}=f_{\pm}$ denotes the
`martingale transform' introduced above. Then the family of solutions
$\left\{  u_{f_{\alpha}}\right\}  $ satisfies the average inequality,%
\begin{equation}
\mathbb{E}_{2^{\mathcal{G}}}\left\Vert u_{f_{\alpha}}\right\Vert
_{L^{q}\left(  \mathbb{R}^{d}\right)  }\leq C_{d,\sigma,p,q}\left\Vert
f\right\Vert _{W^{\sigma,p}\left(  \mathbb{R}^{d}\right)  }%
,\ \ \ \ \ \text{for all }f\in W^{\sigma,p}\left(  \mathbb{R}^{d}\right)
,\label{reg}%
\end{equation}
provided $\sigma>0$ and%
\begin{align}
\frac{2d}{d+1}  & <p\leq2,\label{index conditions}\\
\frac{d+1}{d-1}p  & <q<\frac{d+1}{d-1-p\sigma}p.\nonumber
\end{align}
In particular, for every $q>\frac{2d}{d-1}$, there exists $p$ and $\sigma$
such that $\frac{2d}{d+1}<p\leq2$ and $\sigma>0$ and (\ref{reg}) holds.
Moreover, as $q\searrow\frac{2d}{d-1}$, we must have $p\searrow\frac{2d}{d+1}$
and $\sigma\searrow0$.
\end{theorem}

Lemma \ref{necc} below shows that the exponent range (\ref{index conditions})
in the theorem is sharp up to endpoints. The case $d=3$ of the theorem holds
in the stronger deterministic sense by the recent result of OpenAI
\cite[Corollary 1.2]{Ope}.

\subsection{Necessity by parabolic scaling of the square function
reformulation}

Now assume we have the square function inequality,%
\begin{equation}
\left\Vert \left(  \sum_{I\in\mathcal{G}}\left\vert \mathcal{E}\mathbf{1}%
_{U}\bigtriangleup_{I;\kappa}^{\eta}f\right\vert ^{2}\right)  ^{\frac{1}{2}%
}\right\Vert _{L^{q}\left(  \mathbb{R}^{d}\right)  }\leq C\left\Vert
f\right\Vert _{W^{\sigma,p}(\mathbb{R}^{d-1})}=C\left\Vert (I-\Delta
)^{\frac{\sigma}{2}}f\right\Vert _{L^{p}(\mathbb{R}^{d-1})}\label{good-est}%
\end{equation}
and for convenience in notation we set $u=\operatorname*{S}^{\mathcal{F}%
}f\equiv\left(  \sum_{I\in\mathcal{G}}\left\vert \mathcal{E}\mathbf{1}%
_{U}\bigtriangleup_{I;\kappa}^{\eta}f\right\vert ^{2}\right)  ^{\frac{1}{2}}$.
Then consider the parabolic dilation $u_{\lambda}(t,x)=u(\lambda^{2}t,\lambda
x)$ and the standard dilation $f_{\lambda}(x)=f(\lambda x)$. Substituting
$u_{\lambda}$ and $f_{\lambda}$ into (\ref{good-est}) and gathering the powers
of $\lambda$ together, we get
\begin{equation}
\Vert u\Vert_{L^{q}(\mathbb{R}^{d})}\leq C\lambda^{\frac{d+1}{q}-\frac{d-1}%
{p}}\Vert(I-\lambda^{2}\Delta)^{\frac{\sigma}{2}}f\Vert_{L^{p}(\mathbb{R}%
^{d-1})}.\label{bd-mixed}%
\end{equation}
As $\lambda\rightarrow0$ we have for the right-hand side
\[
\lambda^{\frac{d+1}{q}-\frac{d-1}{p}}\Vert(I-\lambda^{2}\Delta)^{\frac{\sigma
}{2}}f\Vert_{L^{p}(\mathbb{R}^{d-1})}\approx\lambda^{\frac{d+1}{q}-\frac
{d-1}{p}}\Vert f\Vert_{L^{p}(\mathbb{R}^{d-1})},
\]
so we must have $\frac{d+1}{q}\leq\frac{d-1}{p}$ to satisfy (\ref{bd-mixed}).
If $\lambda\rightarrow\infty$ we have
\[
\lambda^{\frac{d+1}{q}-\frac{d-1}{p}}\Vert(I-\lambda^{2}\Delta)^{\frac{\sigma
}{2}}f\Vert_{L^{p}(\mathbb{R}^{d-1})}\approx\lambda^{\frac{d+1}{q}-\frac
{d-1}{p}}\Vert(\lambda^{2}\Delta)^{\frac{\sigma}{2}}f\Vert_{L^{p}%
(\mathbb{R}^{d-1})}=\lambda^{\frac{d+1}{q}-\frac{d-1}{p}+\frac{\sigma}{p}%
}\Vert\Delta^{\frac{\sigma}{2}}f\Vert_{L^{p}(\mathbb{R}^{d-1})},
\]
so we must require $\frac{d+1}{q}\geq\frac{d-1}{p}-\frac{\sigma}{p}$. All in
all, the scaling gives
\[
\frac{d-1}{p}-\frac{\sigma}{p}\leq\frac{d+1}{q}\leq\frac{d-1}{p},
\]
or in terms of $q$
\[
\frac{d+1}{d-1}p\leq q\leq\frac{d+1}{d-1-p\sigma}.
\]
Thus using the equivalence of (\ref{prob ext}) and (\ref{square function}), we
have obtained the following necessary condition.

\begin{lemma}
\label{necc}Let $1<p\leq2$ and $\sigma\geq0$. If the family of solutions
$\left\{  u_{f_{\alpha}}\right\}  $ satisfies%
\[
\mathbb{E}_{2^{\mathcal{G}}}\left\Vert u_{f_{\alpha}}\right\Vert
_{L^{q}\left(  \mathbb{R}^{d}\right)  }\leq C_{d,\sigma,p,q}\left\Vert
f\right\Vert _{W^{\sigma,p}\left(  \mathbb{R}^{d}\right)  }%
,\ \ \ \ \ \text{for all }f\in W^{\sigma,p}\left(  \mathbb{R}^{d}\right)  ,
\]
then
\begin{equation}
\frac{d+1}{d-1}p\leq q\leq\frac{d+1}{d-1-p\sigma}.\label{necc cond}%
\end{equation}

\end{lemma}

Note that the necessary condition (\ref{necc cond}) with $\sigma=0$ is
precisely the condition that $\left(  p^{\prime},q\right)  $ lies on the Knapp
line. Note also that this is guaranteed by our restriction
\[
\frac{d+1}{d-1}p<q<\frac{d+1}{d-1-p\sigma}.
\]

\subsection{Sufficiency by averaging over martingale transforms}

Now we turn to proving sufficiency.

\begin{proof}
[Proof of Theorem \ref{reg}]For convenience we will suppress the martingale
transform notation $f_{\alpha}$ and simply write $f$ with the understanding
that $f=f_{\alpha}=\sum_{I\in\mathcal{G}}\alpha_{I}\left\langle S^{-1}%
\widehat{f},h_{I;\kappa}\right\rangle h_{I;\kappa}^{\eta}$ and $u_{f}%
=u_{f_{\alpha}}$ throughout the following arguments. Now we write%
\begin{align*}
\mathcal{F}_{d-1}f  & =\mathbf{\chi}_{Q}\mathcal{F}_{d-1}f+\sum_{n=1}^{\infty
}\mathbf{\chi}_{2^{n}Q\setminus2^{n-1}Q}\mathcal{F}_{d-1}f\equiv\sum
_{n=0}^{\infty}\widehat{f_{n}}\ ,\\
u_{f}  & =\mathcal{F}_{d}^{-1}\Phi_{\ast}\mathcal{F}_{d-1}f=\mathcal{F}%
_{d}^{-1}\Phi_{\ast}\sum_{n=0}^{\infty}\widehat{f_{n}}=\sum_{n=0}^{\infty
}\mathcal{F}_{d}^{-1}\Phi_{\ast}\widehat{f_{n}}=\sum_{n=0}^{\infty}%
\mathcal{E}\widehat{f_{n}}\ ,
\end{align*}
where $\mathbf{1}_{\mathbb{R}^{d-1}}=\mathbf{\chi}_{Q}\mathcal{F}_{d-1}%
f+\sum_{n=1}^{\infty}\mathbf{\chi}_{2^{n}Q\setminus2^{n-1}Q}$ is a smoothly
adapted partition of unity. Then $\Phi_{\ast}\widehat{f_{0}}$ is a singular
measure in $\mathbb{R}^{d}$ that is absolutely continuous\ with respect to
that portion of surface measure $\sigma_{\mathbb{P}^{d-1}}$ on the paraboloid
that lies above $Q$, and has density $\widehat{f_{0}}\in L^{p^{\prime}}\left(
\Phi\left(  Q\right)  \right)  $.

The probabilistic Fourier extension theorem implies that $\mathcal{F}_{d}%
^{-1}:L^{p^{\prime}}\left(  \mathbf{1}_{\Phi\left(  Q\right)  }\sigma
_{\mathbb{P}^{d-1}}\right)  \rightarrow L^{q}\left(  \mathbb{R}^{d}\right)  $
\emph{on average} if $\frac{1}{p^{\prime}}+\frac{d+1}{d-1}\frac{1}{q}<1$ and
$1\leq p^{\prime}<\frac{2d}{d-1}$, where $Q$ is the unit cube in
$\mathbb{R}^{d-1}$. Thus we conclude that there is a positive constant
$C_{d,p,q}$ such that $\mathcal{F}_{n}^{-1}\Phi_{\ast}\widehat{f_{0}}$
satisfies%
\[
\mathbb{E}_{2^{\mathcal{G}}}\left\Vert \mathcal{F}_{n}^{-1}\Phi_{\ast}%
\widehat{f_{0}}\right\Vert _{L^{q}\left(  \mathbb{R}^{n}\right)  }\leq
C_{d,r,q}\left\Vert \Phi_{\ast}\widehat{f_{0}}\right\Vert _{L^{p^{\prime}%
}\left(  \Phi\left(  Q\right)  \right)  }\approx C_{d,p,q}\left\Vert
\widehat{f_{0}}\right\Vert _{L^{p^{\prime}}\left(  Q\right)  }\lesssim
C_{d,p,q}\left\Vert f_{0}\right\Vert _{L^{p}\left(  Q\right)  },
\]
provided $\frac{1}{p^{\prime}}+\frac{d+1}{d-1}\frac{1}{q}<1$ and $1\leq
p^{\prime}<\frac{2d}{d-1}$.

In order to bound the averages of norms of the functions $\mathcal{F}_{n}%
^{-1}\Phi_{\ast}\widehat{f_{n}}$ when $n\geq1$, we recall

(\textbf{1}) the parabolic rescaling $\delta_{t}^{\operatorname*{para}}$
defined by
\[
\delta_{t}^{\operatorname*{para}}F\left(  \xi_{n},\xi^{\prime}\right)
=F\left(  t^{2}\xi_{n},t\xi^{\prime}\right)  ,
\]
which preserves functions defined on the paraboloid $\mathbb{P}^{n-1}=\left\{
\left(  \xi_{n},\xi^{\prime}\right)  :\xi_{n}=\left\vert \xi^{\prime
}\right\vert ^{2}\right\}  $, and

(\textbf{2}) the standard rescaling $\delta_{t}$ defined in $\mathbb{R}^{k}$
by $\delta_{t}f\left(  \xi^{\prime}\right)  =f\left(  t\xi^{\prime}\right)  $.

Then we have%
\begin{align*}
&  \mathbb{E}_{2^{\mathcal{G}}}\left\Vert \delta_{t}^{\operatorname*{para}%
}\mathcal{E}f\right\Vert _{L^{q}\left(  \mathbb{R}^{d}\right)  }%
^{q}=\mathbb{E}_{2^{\mathcal{G}}}\int\left\vert \delta_{t}%
^{\operatorname*{para}}\mathcal{E}f\right\vert ^{q}d\xi=\mathbb{E}%
_{2^{\mathcal{G}}}\int_{\mathbb{R}}\int_{\mathbb{R}^{d-1}}\left\vert
\mathcal{E}f\left(  t\xi^{\prime},t^{2}\xi_{d}\right)  \right\vert ^{q}%
d\xi^{\prime}d\xi_{d}\\
&  =\mathbb{E}_{2^{\mathcal{G}}}\int_{\mathbb{R}}\int_{\mathbb{R}^{d-1}%
}\left\vert \int e^{-i\left(  \left\vert x\right\vert ^{2},x\right)
\cdot\left(  t^{2}\xi_{d},t\xi^{\prime}\right)  }f\left(  x\right)
dx\right\vert ^{q}d\xi^{\prime}d\xi_{d}=\mathbb{E}_{2^{\mathcal{G}}}%
\int_{\mathbb{R}}\int_{\mathbb{R}^{d-1}}\left\vert \int e^{-i\left(
\left\vert tx\right\vert ^{2},tx\right)  \cdot\left(  \xi_{d},\xi^{\prime
}\right)  }f\left(  x\right)  dx\right\vert ^{q}d\xi^{\prime}d\xi_{d}\\
&  =\mathbb{E}_{2^{\mathcal{G}}}\int_{\mathbb{R}}\int_{\mathbb{R}^{d-1}%
}\left\vert \int e^{-i\left(  \left\vert u\right\vert ^{2},u\right)
\cdot\left(  \xi_{d},\xi^{\prime}\right)  }f\left(  \frac{u}{t}\right)
\frac{du}{t^{d-1}}\right\vert ^{q}d\xi^{\prime}d\xi_{d}=t^{-\left(
d-1\right)  q}\mathbb{E}_{2^{\mathcal{G}}}\left\Vert \mathcal{E}\delta
_{\frac{1}{t}}f\right\Vert _{L^{q}\left(  \mathbb{R}^{d}\right)  }^{q}.
\end{align*}
We also have%
\begin{align*}
\left\Vert \delta_{t}^{\operatorname*{para}}F\right\Vert _{L^{q}\left(
\mathbb{R}^{d}\right)  }^{q} &  =\int_{\mathbb{R}^{d}}\left\vert \delta
_{t}^{\operatorname*{para}}F\left(  \xi\right)  \right\vert ^{q}d\xi
=\int_{\mathbb{R}}\int_{\mathbb{R}^{d-1}}\left\vert F\left(  t^{2}\xi_{d}%
,t\xi^{\prime}\right)  \right\vert ^{q}d\xi^{\prime}d\xi_{d}\\
&  =\int_{\mathbb{R}}\int_{\mathbb{R}^{d-1}}\left\vert F\left(  \eta_{d}%
,\eta^{\prime}\right)  \right\vert ^{q}\frac{d\eta^{\prime}}{t^{d-1}}%
\frac{d\eta_{d}}{t^{2}}=\frac{1}{t^{d+1}}\left\Vert F\right\Vert
_{L^{q}\left(  \mathbb{R}^{d}\right)  }^{q}%
\end{align*}
and%
\begin{align*}
\left\Vert \delta_{\frac{1}{t}}H\right\Vert _{L^{p}\left(  \mathbb{R}%
^{d-1}\right)  }^{p} &  =\int_{\mathbb{R}^{d-1}}\left\vert \delta_{\frac{1}%
{t}}H\left(  x\right)  \right\vert ^{p}dx=\int_{\mathbb{R}^{d-1}}\left\vert
H\left(  \frac{x}{t}\right)  \right\vert ^{p}dx\\
&  =\int_{\mathbb{R}^{d-1}}\left\vert H\left(  y\right)  \right\vert
^{p}t^{d-1}dy=t^{d-1}\left\Vert H\right\Vert _{L^{p}\left(  \mathbb{R}%
^{d-1}\right)  }^{p}.
\end{align*}

Now we apply $\delta_{2^{n}}$ to $\widehat{f_{n}}=\mathbf{\chi}_{2^{n}%
Q\setminus2^{n-1}Q}\mathcal{F}_{d-1}f\in L^{p^{\prime}}\left(  2^{n}%
Q\setminus2^{n-1}Q\right)  $ and note that%
\[
\left\Vert \widehat{f_{n}}\right\Vert _{L^{p^{\prime}}\left(  2^{n}%
Q\setminus2^{n-1}Q\right)  }^{p^{\prime}}=2^{n\left(  d-1\right)  }\left\Vert
\delta_{2^{n}}\widehat{f_{n}}\right\Vert _{L^{p^{\prime}}\left(  Q\right)
}^{p^{\prime}},
\]
where $\delta_{2^{n}}\widehat{f_{n}}\left(  x\right)  =\widehat{f_{n}}\left(
2^{n}x\right)  $ is supported in $Q$. We also apply $\delta_{t}%
^{\operatorname*{para}}$ to $\mathcal{E}\widehat{f_{n}}$ to obtain%
\[
\mathbb{E}_{2^{\mathcal{G}}}\left\Vert \delta_{2^{-n}}^{\operatorname*{para}%
}\mathcal{E}\widehat{f_{n}}\right\Vert _{L^{q}\left(  \mathbb{R}^{d}\right)
}^{q}=2^{n\left(  d-1\right)  q}\mathbb{E}_{\mathbb{A}}\left\Vert
\mathcal{E}\delta_{2^{n}}\widehat{f_{n}}\right\Vert _{L^{q}\left(
\mathbb{R}^{d}\right)  }^{q}%
\]
and then since $\delta_{2^{n}}\widehat{f_{n}}$ is supported in $Q$, the
probabilistic Fourier extension theorem implies%
\begin{align*}
& \mathbb{E}_{2^{\mathcal{G}}}\left\Vert \delta_{2^{-n}}^{\operatorname*{para}%
}\mathcal{E}\widehat{f_{n}}\right\Vert _{L^{q}\left(  \mathbb{R}^{d}\right)
}=2^{n\left(  d-1\right)  }\mathbb{E}_{2^{\mathcal{G}}}\left\Vert
\mathcal{E}\delta_{2^{n}}\widehat{f_{n}}\right\Vert _{L^{q}\left(
\mathbb{R}^{d}\right)  }\leq C_{d,r,q}2^{n\left(  d-1\right)  }\left\Vert
\delta_{2^{n}}\widehat{f_{n}}\right\Vert _{L^{p^{\prime}}\left(
\mathbb{R}^{d}\right)  }\\
& =C_{d,p,q}2^{n\left(  d-1\right)  }2^{-n\frac{d-1}{p^{\prime}}}\left\Vert
\widehat{f_{n}}\right\Vert _{L^{p^{\prime}}\left(  \mathbb{R}^{d-1}\right)
}\leq C_{d,p,q}2^{n\frac{d-1}{p}}\left\Vert f_{n}\right\Vert _{L^{p}\left(
\mathbb{R}^{d-1}\right)  },
\end{align*}
provided $\frac{1}{p^{\prime}}+\frac{d+1}{d-1}\frac{1}{q}<1$ and $1\leq
p^{\prime}<\frac{2d}{d-1}$.

Thus we have%
\[
\mathbb{E}_{2^{\mathcal{G}}}\left\Vert \mathcal{E}\widehat{f_{n}}\right\Vert
_{L^{q}\left(  \mathbb{R}^{d}\right)  }=2^{-n\frac{d+1}{q}}\mathbb{E}%
_{2^{\mathcal{G}}}\left\Vert \delta_{2^{-n}}^{\operatorname*{para}}%
\mathcal{E}\widehat{f_{n}}\right\Vert _{L^{q}\left(  \mathbb{R}^{d}\right)
}\leq C_{d,p,q}2^{n\left(  \frac{d-1}{p}-\frac{d+1}{q}\right)  }\left\Vert
f_{n}\right\Vert _{L^{p}\left(  \mathbb{R}^{d-1}\right)  },
\]
and as a consequence%
\[
\mathbb{E}_{2^{\mathcal{G}}}\left\Vert \mathcal{E}\widehat{f_{n}}\right\Vert
_{L^{q}\left(  \mathbb{R}^{d}\right)  }\leq C_{d,r,q}2^{n\left(  \frac{d-1}%
{p}-\frac{d+1}{q}\right)  }2^{-n\sigma}\left\Vert f_{n}\right\Vert
_{W^{\sigma,p}\left(  \mathbb{R}^{d-1}\right)  }.
\]
Indeed,%
\[
\left\Vert f_{n}\right\Vert _{W^{\sigma,p}\left(  \mathbb{R}^{d-1}\right)
}=\left\Vert \left(  I-\bigtriangleup\right)  ^{\frac{\sigma}{2}}%
f_{n}\right\Vert _{L^{p}\left(  \mathbb{R}^{d-1}\right)  }\approx2^{n\sigma
}\left\Vert f_{n}\right\Vert _{L^{p}\left(  \mathbb{R}^{d-1}\right)  }%
\]
since
\begin{align*}
\left(  \left(  I-\bigtriangleup\right)  ^{\frac{\sigma}{2}}f_{n}\right)
^{\wedge}\left(  \xi\right)    & =\left(  1+\left\vert \xi\right\vert
^{2}\right)  ^{\frac{\sigma}{2}}\widehat{f_{n}}\left(  \xi\right)  =\left(
1+\left\vert \xi\right\vert ^{2}\right)  ^{\frac{\sigma}{2}}\mathbf{\chi
}_{2^{n}Q\setminus2^{n-1}Q}\mathcal{F}_{d-1}f\left(  \xi\right)  \\
& \approx2^{n\sigma}\mathbf{\chi}_{2^{n}Q\setminus2^{n-1}Q}\mathcal{F}%
_{d-1}f\left(  \xi\right)  \approx2^{n\sigma}\widehat{f_{n}}\left(
\xi\right)
\end{align*}
implies%
\[
\left\Vert \left(  I-\bigtriangleup\right)  ^{\frac{\sigma}{2}}f_{n}%
\right\Vert _{L^{p}\left(  \mathbb{R}^{d-1}\right)  }=\left\Vert
\mathcal{F}_{d-1}\left(  \left(  I-\bigtriangleup\right)  ^{\frac{\sigma}{2}%
}f_{n}\right)  ^{\wedge}\right\Vert _{L^{p}\left(  \mathbb{R}^{d-1}\right)
}\lesssim\left\Vert 2^{n\sigma}f_{n}\right\Vert _{L^{p}\left(  \mathbb{R}%
^{d-1}\right)  }=2^{n\sigma}\left\Vert f_{n}\right\Vert _{L^{p}\left(
\mathbb{R}^{d-1}\right)  },
\]
by Littlewood-Paley theory, and the reverse inequality is an easy exercise.

So now we conclude that
\begin{align*}
& \mathbb{E}_{2^{\mathcal{G}}}\left\Vert u_{f}\right\Vert _{L^{q}\left(
\mathbb{R}^{d}\right)  }=\mathbb{E}_{2^{\mathcal{G}}}\left\Vert \mathcal{E}%
\widehat{f}\right\Vert _{L^{q}\left(  \mathbb{R}^{d}\right)  }=\mathbb{E}%
_{2^{\mathcal{G}}}\left\Vert \sum_{n=0}^{\infty}\mathcal{E}\widehat{f_{n}%
}\right\Vert _{L^{q}\left(  \mathbb{R}^{d}\right)  }\leq\sum_{n=0}^{\infty
}\mathbb{E}_{2^{\mathcal{G}}}\left\Vert \mathcal{E}\widehat{f_{n}}\right\Vert
_{L^{q}\left(  \mathbb{R}^{d}\right)  }\\
& \leq C_{d,p,q}\sum_{n=0}^{\infty}2^{n\left(  \frac{d-1}{p}-\frac{d+1}%
{q}\right)  }2^{-n\sigma}\left\Vert f_{n}\right\Vert _{W^{\sigma,p}\left(
\mathbb{R}^{d-1}\right)  }\\
& \leq C_{d,p,q}\left(  \sum_{n=0}^{\infty}2^{np^{\prime}\left(  \frac{d-1}%
{p}-\frac{d+1}{q}-\frac{\sigma}{r}\right)  }\right)  ^{\frac{1}{p^{\prime}}%
}\left(  \sum_{n=0}^{\infty}2^{-n\frac{\sigma}{r^{\prime}}p}\left\Vert
f_{n}\right\Vert _{W^{p,\sigma}\left(  \mathbb{R}^{d-1}\right)  }^{p}\right)
^{\frac{1}{p}}\\
& \leq C_{d,\sigma,p,q}\left(  \sum_{n=0}^{\infty}2^{-n\frac{\sigma}%
{r^{\prime}}p}\left\Vert \left(  I-\bigtriangleup\right)  ^{\frac{\sigma}{2}%
}f_{n}\right\Vert _{L^{p}\left(  \mathbb{R}^{d-1}\right)  }^{p}\right)
^{\frac{1}{p}},
\end{align*}
which is at most%
\begin{align*}
& C_{d,\sigma,p,q}\left(  \int\sum_{n=0}^{\infty}2^{-n\frac{\sigma}{r^{\prime
}}p}\left\vert \left(  I-\bigtriangleup\right)  ^{\frac{\sigma}{2}}%
f_{n}\left(  \xi\right)  \right\vert ^{p}\right)  ^{\frac{1}{p}}\\
& \leq C_{d,\sigma,p,q}\left(  \int\left(  \sum_{n=0}^{\infty}\left\vert
\left(  I-\bigtriangleup\right)  ^{\frac{\sigma}{2}}f_{n}\left(  \xi\right)
\right\vert ^{2}\right)  ^{\frac{1}{2}}\right)  ^{\frac{1}{p}}\\
& =C_{d,\sigma,p,q}^{\prime}\left(  \int\left(  \sum_{n=0}^{\infty}\left\vert
\mathsf{P}_{n}\left(  I-\bigtriangleup\right)  ^{\frac{\sigma}{2}}f\left(
\xi\right)  \right\vert ^{2}\right)  ^{\frac{1}{2}}\right)  ^{\frac{1}{p}}\\
& \leq C_{d,\sigma,p,q}^{\prime\prime}\left\Vert \left(  I-\bigtriangleup
\right)  ^{\frac{\sigma}{2}}f\right\Vert _{L^{p}\left(  \mathbb{R}^{d}\right)
}=C_{d,\sigma,p,q}^{\prime\prime}\left\Vert f\right\Vert _{W^{\sigma,r}\left(
\mathbb{R}^{d}\right)  },
\end{align*}
where $\mathsf{P}_{n}$ is the Fourier multiplier with symbol $\mathbf{\chi
}_{2^{n}Q\setminus2^{n-1}Q}$, by Littlewood-Paley theory, provided $\frac
{d-1}{p}-\frac{\sigma}{r}<\frac{d+1}{q}$, i.e. $q<\frac{d+1}{\frac{d-1}%
{r}-\frac{\sigma}{r}}=\frac{d+1}{d-1-\frac{p\sigma}{r}}p$. Since the
probabilistic Fourier extension theorem requires $2\leq p^{\prime}<\frac
{2d}{d-1}$ and $\frac{1}{p^{\prime}}+\frac{d+1}{d-1}\frac{1}{q}<1$, i.e.
$\frac{2d}{d+1}<p\leq2$ and $q>\frac{d+1}{d-1}p$, it suffices to take%
\[
\frac{2d}{d+1}<p\leq2\text{ and }\frac{d+1}{d-1}p<q<\frac{d+1}{d-1-\frac
{p\sigma}{r}}p,
\]
for any $r>1$, and this completes the proof of the first part of Theorem
\ref{Thm reg}. The second part is immediate since given
\[
\sigma>0\text{ and }q>\frac{2d}{d-1}<q\leq2\frac{d+1}{d-1-\sigma},
\]
which is not a vacuous condition for $\sigma\geq0$, then
(\ref{index conditions}) holds if we choose $p$ to satisfy%
\[
\max\left\{  \frac{2d}{d+1},q\frac{d-1}{d+1+\frac{q\sigma}{r}}\right\}
<p<q\frac{d-1}{d+1}\text{ and }p\leq2,
\]
which is again not a vacuous condition.
\end{proof}

\section{Comparison with some other results}

Duvan Cardona \cite{Car} obtains the following deterministic estimate,
\begin{align}
& \left\Vert \left\Vert u_{f}\right\Vert _{L_{t}^{s}\left[  0,1\right]
}\right\Vert _{L_{x}^{p^{\prime}}\left(  \mathbb{R}^{n-1}\right)  }%
\lesssim\left\Vert f\right\Vert _{L^{r}\left(  \mathbb{R}^{n-1}\right)
},\label{Car gen}\\
1  & \leq p\leq2\text{, }1\leq s\leq p^{\prime}\text{ and }\frac{1}{p}%
-\frac{1}{2}<\frac{1}{\left(  n-1\right)  s}.\nonumber
\end{align}
The special case $s=p^{\prime}$ is
\begin{align}
& \left\Vert u_{f}\right\Vert _{L\widetilde{^{q}}\left(  \left[  0,1\right]
\times\mathbb{R}^{n-1}\right)  }\lesssim\left\Vert f\right\Vert _{L^{r}\left(
\mathbb{R}^{n-1}\right)  },\label{Car spec}\\
\text{for }\widetilde{q}  & =\frac{r}{r-1}\text{ and }\frac{2n}{n+1}%
<r\leq2,\nonumber
\end{align}
which should be compared with our \emph{average} conjecture,
\begin{align*}
& \mathbb{E}_{\mathbb{A}}\left\Vert u_{f_{\alpha}}\right\Vert _{L^{q}\left(
\mathbb{R}\times\mathbb{R}^{n-1}\right)  }\lesssim\left\Vert f\right\Vert
_{L^{r}\left(  \mathbb{R}^{n-1}\right)  },\\
\text{for }q  & =\frac{n+1}{n-1}r\text{ and }\frac{2n}{n+1}<r\leq2.
\end{align*}

\begin{enumerate}
\item First, note that $\widetilde{q}=q$ if and only if $r=\frac{2n}{n+1}$,
and hence $\widetilde{q}<q$ for $\frac{2n}{n+1}<r\leq2$.

\item Second, Cardona integrates over a substrip of $\mathbb{R}\times
\mathbb{R}^{n-1}$ on the left hand side.

\item Third, if we restrict our integration to the same substrip, and this
with (\ref{Car spec}), then we obtain the average inequality%
\[
\mathbb{E}_{\mathbb{A}}\left\Vert u_{f_{\alpha}}\right\Vert _{L^{q}\left(
\left[  0,1\right]  \times\mathbb{R}^{n-1}\right)  }\lesssim\left\Vert
f\right\Vert _{L^{r}\left(  \mathbb{R}^{n-1}\right)  },
\]
which is a strict improvement of the indices in (\ref{Car spec}), but of
course we only obtain a martingale averaged inequality. In any event, our
result doesn't follow from (\ref{Car gen}).
\end{enumerate}

\end{document}